\documentclass[12pt]{amsart} 

\usepackage{amscd} 
\usepackage{amsmath} 
\usepackage{amssymb} 
\usepackage{amsthm}
\usepackage{bigdelim}
\usepackage{color} 
\usepackage{enumerate}
\usepackage{mathrsfs}
\usepackage{multirow}
\usepackage[all]{xy} 
\usepackage{hyperref}
\newtheorem{thm}{Theorem}[section]

\newtheorem{lem}[thm]{Lemma}
\theoremstyle{definition} 
\newtheorem{defn}[thm]{Definition}

\theoremstyle{remark}

\newtheorem{ques}[thm]{Question}

\newtheorem*{ack}{Acknowledgements}

\title{Notes on Frobenius stable direct images}
\author{Sho Ejiri}
\address{Department of Mathematics, Graduate School of Science, Osaka Metropolitan University, Osaka City, Osaka 558-8585, Japan}
\email{shoejiri.math@gmail.com}
\begin{document}
\maketitle
\markboth{SHO EJIRI}{Notes on Frobenius stable direct images}
\begin{abstract}
In this note, 
we prove the coherence of Frobenius stable direct images in a new case. 
We also show a generation theorem regarding to it. 
Furthermore, we prove a corresponding theorem in characteristic zero. 
\end{abstract}
\section{Introduction}
Let $X$ be a normal projective variety over 
an algebraically closed field of positive characteristic,
let $\Delta$ be an effective $\mathbb Q$-Weil divisor on $X$,  
and let $M$ be a Cartier divisor on $X$. 
In~\cite{Sch14} (cf. \cite{Pat14}), Schwede introduced the subspace 
$$
S^0(X,\sigma(X,\Delta) \otimes \mathcal O_X(M))
\subseteq
H^0(X, \mathcal O_X(M)), 
$$  
which is defined as the stable image of 
the trace maps of iterated Frobenius morphisms. 
This notion was relativized by Hacon and Xu~\cite{HX15}
to establish the three dimensional minimal model program 
in positive characteristic (Definition~\ref{defn:S0}). 
For a morphism $f:X\to Y$ to a variety $Y$, they define the subsheaf 
$$
S^0f_*(\sigma(X,\Delta) \otimes \mathcal O_X(M)) 
\subseteq
f_* \mathcal O_X(M) 
$$
by a way similar to that of $S^0(X,\sigma(X,\Delta) \otimes \mathcal O_X(M))$. 
From the definition, we cannot not see whether or not the sheaf 
$
S^0f_*(\sigma(X,\Delta)\otimes\mathcal O_X(M))
$ 
is coherent. 
The coherence is proved in \cite[Proposition~2.15]{HX15}, 
under the assumption that $M-(K_X+\Delta)$ is $f$-ample.  
In this note, we show the coherence of the sheaf under a weaker assumption:
\begin{thm} \label{thm:main-p1}
Let the base field be an $F$-finite field. 
Let $X$ be a normal projective variety 
and let $\Delta$ be an effective $\mathbb Q$-Weil divisor on $X$ such that 
$i(K_X+\Delta)$ is Cartier for an integer $i>0$ not divisible by $p$. 
Let $f:X\to Y$ be a morphism to a projective variety $Y$ of dimension $n$. 
Let $M$ be a Cartier divisor on $X$. 
If $M-(K_X+\Delta)$ is relatively semi-ample over 
an open subset $V\subset Y$, then 
$$
\mathrm{Im}\left(f_* \phi_{(X,\Delta)}^{(e)}(M) \right) |_V
=
S^0f_* (\sigma(X,\Delta) \otimes \mathcal O_X(M)) |_V
$$ 
for $e$ large and divisible enough. In particular, 
$
S^0f_* (\sigma(X,\Delta) \otimes \mathcal O_X(M)) 
$
is coherent over $V$. 
\end{thm}
We also prove a generation theorem on 
$S^0f_*(\sigma(X,\Delta)\otimes \mathcal O_X(M))$. 
\begin{thm} \label{thm:main-p2}
With the notation of Theorem~\ref{thm:main-p1},
if $M-(K_X+\Delta)$ is nef and $f$-semi-ample, then the sheaf 
$$
\left( S^0f_*(\sigma(X,\Delta) \otimes \mathcal O_X(M)) \right) 
\otimes \mathcal L^n \otimes \mathcal A
$$
is generated by its global sections for ample line bundles $\mathcal L$ and $\mathcal A$ with $|\mathcal L|$ free.
\end{thm}
For a related result, see \cite[Theorem~1.11]{Zha19s}.
Furthermore, we show a corresponding theorem in characteristic zero 
to Theorem~\ref{thm:main-p2}. 
\begin{thm} \label{thm:main-0}
Let the base field be a field of characteristic zero. 
Let $X$ be a normal projective variety 
and let $\Delta$ be an effective $\mathbb Q$-Weil divisor on $X$ 
such that $K_X+\Delta$ is $\mathbb Q$-Cartier. 
Let $f:X\to Y$ be a morphism to a projective variety $Y$ of dimension $n$. 
Let $M$ be a Cartier divisor on $X$. 
If $M-(K_X+\Delta)$ is nef and $f$-semi-ample, then 
\begin{align*}
f_*(\mathcal J(X,\Delta)(M)) \otimes \mathcal L^n \otimes \mathcal A
\quad \textup{and} \quad 
f_*(\mathcal J_{\mathrm{NLC}}(X,\Delta)(M)) 
\otimes \mathcal L^n \otimes \mathcal A
\end{align*}
are generated by its global sections for ample line bundles 
$\mathcal L$ and $\mathcal A$ with $|\mathcal L|$ free. 
In particular, if, furthermore, $(X,\Delta)$ is log canonical, 
then $f_*\mathcal O_X(M) \otimes \mathcal L^n \otimes \mathcal A$ is generated by its global sections. 
\end{thm}
Here, $\mathcal J(X,\Delta)$ (resp. $\mathcal J_{\mathrm{NLC}}(X,\Delta)$) 
is the multiplier ideal (resp. the non-lc ideal sheaf)
associated to $(X,\Delta)$ (Definition~\ref{defn:non-lc}). 
Theorem~\ref{thm:main-0} should be compared with~\cite[Corollary~1.7]{FM21i}.
Also, for several results related to Theorem~\ref{thm:main-0}, 
see \cite[Section~9]{Fuj19}.
 
Theorem~\ref{thm:main-0} does not hold in positive characteristic. 
Indeed, Moret-Bailly~\cite{MB81} constructed a semi-stable fibration 
$g:S\to \mathbb P^1$ of characteristic $p>0$, 
where $S$ is a smooth projective surface, 
such that 
$
g_*\omega_S \otimes \mathcal O(2) 
\cong \mathcal O(-1)\oplus\mathcal O(p)
$
(see also \cite[Proposition~3.16]{SZ20}).

The above theorems are proved in Section~\ref{section:proofs}. 
In Section~\ref{section:question}, we consider a question that generalizes 
Fujita's freeness conjecture. 
\begin{ack}
The author wises to express his thanks to Professors Osamu Fujino, Masataka Iwai and Shin-ichi Matsumura for helpful comments. 
He is grateful to Professor Shunsuke Takagi for a fruitful question. 
\end{ack}
\section{Terminologies and definitions} \label{section:definitions}
In this section, we define some terminologies and notions. 
 
Let $k$ be a field of characteristic $p\ge 0$. 
By a \textit{variety} we mean an integral separated scheme of finite type over $k$. 

Let $X$ be a normal variety and 
let $\Delta=\sum_{i=1}^d \delta_i\Delta_i$ be a $\mathbb Q$-Weil divisor on $X$, 
where each $\Delta_i$ is a prime divisor. 
We define the \textit{round down} $\lfloor \Delta \rfloor$ 
(resp. \textit{round up} $\lceil \Delta \rceil$) 
of $\Delta$ by 
$ \lfloor \Delta \rfloor := \sum_{i=1}^d \lfloor\delta_i\rfloor \Delta_i $
(resp. $ \lceil \Delta \rceil := \sum_{i=1}^d \lceil\delta_i\rceil \Delta_i $). 
Also, we use the following notation: 
\begin{align*}
\{\Delta\}:=\Delta-\lfloor\Delta\rfloor; \qquad 
\Delta^{=1} := \sum_{\delta_i=1} \Delta_i; \qquad
\Delta^{>1} := \sum_{\delta_i>1} \Delta_i; \qquad
\Delta^{<1} := \sum_{\delta_i<1} \Delta_i. 
\end{align*}

Assume that $p>0$. Let $X$ be a variety. Let $F_X^e:X\to X$ denote 
the $e$-times iterated absolute Frobenius morphism of $X$. 

Let $k$ be an $F$-field field, 
i.e., a field of characteristic $p>0$ with $[k:k^p]<+\infty$. 
Let $X$ be a normal variety 
and let $\Delta$ be an effective $\mathbb Q$-Weil divisor on $X$ 
such that $i(K_X+\Delta)$ is Cartier for an integer $i>0$ not divisible by $p$. 
Let $M$ be a Cartier divisor on $X$. 
Let $f:X\to Y$ be a projective morphism to a variety $Y$. 
We define $S^0f_*(\sigma(X,\Delta) \otimes \mathcal O_X(M))$. 
Let $e_0$ be the smallest positive integer such that $i|(p^{e_0}-1)$. 
For each $e\ge1$ with $e_0|e$, 
applying $\mathcal Hom_{\mathcal O_X}((?),\mathcal O_X(M))$ to the composite of
$$
\mathcal O_X 
\xrightarrow{{F_X^e}^\sharp} {F_X^e}_*\mathcal O_X
\hookrightarrow {F_X^e}_* \mathcal O_X((p^e-1)\Delta), 
$$
we obtain the morphism 
$$
\phi^{(e)}_{(X,\Delta)}(M):{F_X^e}_* \mathcal O_X((1-p^e)(K_X+\Delta) +p^eM)
\to \mathcal O_X(M) 
$$
by the Grothendieck duality. Pushing this forward by $f$, we get 
$$
f_*\phi^{(e)}_{(X,\Delta)}(M):{F_Y^e}_*f_*\mathcal O_X((1-p^e)(K_X+\Delta) +p^eM)
\to f_* \mathcal O_X(M). 
$$
Note that $f_*{F_X^e}_* = {F_Y^e}_* f_*$. 
By the construction, we see that $f_*\phi^{(e)}_{(X,\Delta)}(M)$ factors through 
$f_*\phi^{(e')}_{(X,\Delta)}(M)$ for $e\ge e'\ge 1$ with $e_0|e$ and $e_0|e'$, 
so 
$$
\mathrm{Im}\left(f_*\phi^{(e)}_{(X,\Delta)}(M)\right)
\subseteq \mathrm{Im}\left(f_*\phi^{(e')}_{(X,\Delta)}(M)\right). 
$$
\begin{defn} \label{defn:S0}
With the above notation, we define 
$$
S^0f_*(\sigma(X,\Delta) \otimes \mathcal O_X(M))
:=\bigcap_{e\ge1,~e_0|e} \mathrm{Im}\left(f_*\phi^{(e)}_{(X,\Delta)}(M)\right)
\subseteq f_*\mathcal O_X(M). 
$$
\end{defn}
We cannot see from the definition whether or not 
$S^0f_*(\sigma(X,\Delta)\otimes \mathcal O_X(M))$ is coherent. 
 
Next, we define the multiplier ideal sheaf $J(X,\Delta)$ and 
the non-lc ideal sheaf $\mathcal J_{\mathrm{NLC}}(X,\Delta)$. 
\begin{defn} \label{defn:non-lc}
Let the base field be a field of characteristic zero.  
Let $X$ be a normal projective variety and 
let $\Delta$ be an effective $\mathbb Q$-Weil divisor such that 
$K_X+\Delta$ is $\mathbb Q$-Cartier. 
Let $\pi:Z\to X$ be a resolution of $X$ with 
$
K_Z+\Delta_Z = \pi^*(K_X+\Delta)
$
such that $\mathrm{Supp}(\Delta_Z)$ is simple normal crossing. 
\begin{itemize}
\item We define the \textit{multiplier ideal sheaf} $\mathcal J(X,\Delta)$ by 
$$
\mathcal J(X,\Delta):= \pi_*\mathcal O_Z\left(-\lfloor \Delta_Z\rfloor\right). 
$$
\item (\cite[Definition~2.1]{Fuj10}) 
We define the \textit{non-lc ideal sheaf} $\mathcal J_{\mathrm{NLC}}(X,\Delta)$ 
by 
$$
\mathcal J_{\mathrm{NLC}}(X,\Delta)
:=\pi_*\mathcal O_Z \left( \lceil -(\Delta_Z^{<1}) \rceil 
-\lfloor \Delta_Z^{>1} \rfloor \right)
=\pi_*\mathcal O_Z\left( -\lfloor \Delta_Z \rfloor +\Delta_Z^{=1}\right). 
$$
\end{itemize}
\end{defn}
By \cite[Proposition~2.6]{Fuj10}, 
we see that $\mathcal J_{\mathrm{NLC}}(X,\Delta)$ is 
independent of the choice of the resolution $\pi:Z\to X$, 
so $\mathcal J_{\mathrm{NLC}}(X,\Delta)$ is well-defined. 
We see from the definition that 
$\mathcal J(X,\Delta) = \mathcal O_X$
(resp. $\mathcal J_{\mathrm{NLC}}(X,\Delta) =\mathcal O_X$)
if and only if $(X,\Delta)$ is Kawamata log terminal (resp. log canonical). 
\section{Proofs of the theorems} \label{section:proofs}
We first prove Theorem~\ref{thm:main-p1}. 
\begin{proof}[Proof of Theorem~\ref{thm:main-p1}]
We prove (1). We use the notation in Section~\ref{section:definitions}. 
Put $\mathcal I^{(e)}:=\mathrm{Im}\left(f_*\phi^{(e)}_{(X,\Delta)}(M)\right)$ 
for each $e\ge 1$ with $e_0|e$. 
We show that there is an ample line bundle $\mathcal L$ on $Y$ 
such that $\mathcal I^{(e)}\otimes\mathcal L$ 
is globally generated on $V$ for $e\gg 0$ with $e_0|e$. 
If this holds, then for each $e\ge e'\gg 0$ with $e_0|e$ and $e_0|e'$, 
we have 
$$
\xymatrix{ 
H^0(Y, \mathcal I^{(e')} \otimes \mathcal L) \otimes_k \mathcal O_Y 
\ar@{->}[r] & \mathcal I^{(e')} \otimes \mathcal L
\\ H^0(Y, \mathcal I^{(e)} \otimes \mathcal L) \otimes_k \mathcal O_Y
\ar@{->}[r] \ar@{=}[u] & \mathcal I^{(e)} \otimes \mathcal L, \ar@{^(->}[u]
}
$$
where $k$ is the base field and the horizontal arrows are surjective on $V$. 
Note that the equality follows from the fact that 
the space of global sections is of finite dimension. 
Hence, 
$
(\mathcal I^{(e)} \otimes \mathcal L)|_V
=(\mathcal I^{(e')} \otimes \mathcal L)|_V,
$ 
which implies that $\mathcal I^{(e)}|_V=\mathcal I^{(e')}|_V$. 
 
Set $N:=M-K_X-\Delta$. Then $iN$ is Cartier. 
Since $N$ is relatively semi-ample over $V$, 
there is $d\ge 2$ with $i|d$ and $m_0\ge1$ such that 
the natural morphism 
$$
f_* \mathcal O_X(dN)
\otimes 
f_* \mathcal O_X(mN +M)
\to 
f_* \mathcal O_X((d+m)N +M)
$$
is surjective on $V$ for $m\ge m_0$ with $i|m$. 
Let $\mathcal L$ be an ample line bundle on $Y$ such that $\mathcal L$ and 
$
f_*\mathcal O_X(dN) \otimes \mathcal L
$
are globally generated.
We prove that $\mathcal I^{(e)}\otimes \mathcal L^{n+1}$ is 
globally generated on $V$ for each $e\gg0$ with $e_0|e$. 
Let $q_e$ and $r_e$ be integers such that 
$p^e-1=q_ed+r_e$ and $m_0\le r_e < m_0 +d$. 
Note that if $e_0|e$, then $i|r_e$. 
Put $\mathcal G:=\bigoplus_{m_0\le r<m_0+d,~i|r} f_*\mathcal O_X(rN+M)$. 
We then have the following sequence of morphisms that are surjective on $V$ 
for each $e\ge1$ with $e_0|e$:
\begin{align*}
& \mathcal I^{(e)} \otimes \mathcal L^{n+1}
\\ \xleftarrow{\left( f_*\phi^{(e)}_{(X,\Delta)}(M)\right) \otimes \mathcal L^{n+1}} &
{F_Y^e}_* \left( f_*\mathcal O_X((p^e-1)N +M) \otimes \mathcal L^{p^e(n+1)} 
\right) 
\\ \leftarrow & 
{F_Y^e}_* \left( S^{q_e}(f_*\mathcal O_X(dN)) \otimes f_*\mathcal O_X(r_eN +M) 
\otimes \mathcal L^{p^e(n+1)} \right)
\\ \cong & 
{F_Y^e}_* \left( S^{q_e}(f_*\mathcal O_X(dN) \otimes \mathcal L) 
\otimes f_*\mathcal O_X(r_eN +M) \otimes \mathcal L^{p^e(n+1)-q_e} \right) 
\\ \twoheadleftarrow & 
{F_Y^e}_* \left( \left( \bigoplus \mathcal O_Y \right) 
\otimes f_*\mathcal O_X(r_eN +M) \otimes \mathcal L^{p^e(n+1)-q_e} \right)
\\ \cong & 
\bigoplus {F_Y^e}_* \left( f_*\mathcal O_X(r_eN +M) 
\otimes \mathcal L^{p^e(n+1)-q_e} \right)
\\ \twoheadleftarrow & 
\bigoplus {F_Y^e}_* \left( \mathcal G \otimes\mathcal L^{p^e(n+1)-q_e} \right). 
\end{align*}
Therefore, it is enough to show that 
$
{F_Y^e}_*\left(\mathcal G \otimes \mathcal L^{p^e(n+1)-q_e} \right)
$
is globally generated for $e\gg0$. 
We check that the sheaf is $0$-regular with respect to $\mathcal L$ 
in the sense of Castelnuovo--Mumford (\cite[Theorem~1.8.5]{Laz04I}).
For each $0< j \le n$, we have 
\begin{align*}
H^j\left(Y, {F_Y^e}_*\left(\mathcal G 
\otimes \mathcal L^{p^e(n+1)-q_e} \right) \otimes \mathcal L^{-j} \right)
= & H^j\left(Y, {F_Y^e}_*\left(\mathcal G 
\otimes \mathcal L^{p^e(n+1-j)-q_e} \right) \right)
\\ = & H^j\left(Y, \mathcal G \otimes \mathcal L^{p^e(n+1-j)-q_e} \right). 
\end{align*}
Note that $F_Y^e$ is finite, since $k$ is $F$-finite. 
By $q_e/p^e \xrightarrow{e\to +\infty}1/d<1$,
we get $p^e(n+1-j)-q_e \xrightarrow{e\to +\infty} +\infty$, 
so our claim follows from the Serre vanishing theorem. 
\end{proof}
Next, we show Theorem~\ref{thm:main-p2}. 
To this end, we need Lemma~\ref{lem:glgen2}. 
To prove the lemma, we use the following: 
\begin{lem}[\textup{\cite[Lemma~3.4]{Eji19p}}] \label{lem:glgen1}
Let $f:X\to Y$ be a morphism between projective varieties over a field. 
Let $\mathcal F$ be a coherent sheaf on $X$. 
Let $D$ be an ample Cartier divisor on $X$. 
Then there exists an integer $m_0\ge 1$ such that 
$$
f_* \mathcal F(mD+N)
$$
is generated by its global sections for each $m\ge m_0$ 
and every nef Cartier divisor $N$ on $X$. 
\end{lem}
\begin{lem} \label{lem:glgen2}
Let $f:X\to Y$ be a morphism between projective varieties over a field. 
Let $\mathcal F$ be a coherent sheaf on $X$. 
Let $D$ be a nef and $f$-semi-ample Cartier divisor on $X$. 
Let $\mathcal A$ be an ample line bundle on $Y$. 
Then there exists integers $n_0\ge 1$ and $l_0\ge 1$ such that 
$$
f_* \mathcal F(nD) \otimes \mathcal A^l
$$
is generated by its global sections for all $n\ge n_0$ and $l\ge l_0$. 
\end{lem}
\begin{proof}
Since $D$ is $f$-semi-ample, there are projective morphisms
$\sigma:X\to W$ and $\tau:W\to Y$ with $\tau \circ \sigma=f$ such that 
$mD\sim \sigma^*D'$ for an $m\ge 1$ and 
a nef and $\tau$-ample Cartier divisor $D'$ on $Z$. 
Set $\mathcal F' := \bigoplus_{i=0}^{m-1}\mathcal F(iD)$. 
Replacing $f:X\to Y$, $D$ and $\mathcal F$ 
by $\tau:W\to Y$, $D'$ and $\sigma_*\mathcal F'$, respectively, 
we may assume that $D$ is $f$-ample. 
Since $D$ is nef and $f$-ample, $D+f^*A$ is ample, 
where $A$ is a Cartier divisor with $\mathcal O_Y(A)\cong \mathcal A$. 
Then by Lemma~\ref{lem:glgen1}, there is an $m_0\ge 1$ such that 
$$
f_* \mathcal F(nD+lf^*A) \cong f_* \mathcal F(nD) \otimes \mathcal A^l
$$
is globally generated for each $l,n \ge m_0$. 
\end{proof}
\begin{proof}[Proof of Theorem~\ref{thm:main-p2}]
We use the same notation as that of the proof of Theorem~\ref{thm:main-p1}. 
By Theorem~\ref{thm:main-p1}, it is enough to show that 
$
\mathcal I^{(e)}\otimes \mathcal L^n \otimes \mathcal A
$ 
is globally generated for $e\gg0$ with $e_0|e$. 
By Lemma~\ref{lem:glgen2}, there is an $l\ge1$ such that 
$$
f_*\mathcal O_X(mN+M) \otimes \mathcal A^l
$$
is globally generated for each $m\ge 1$ with $i|m$. 
We then have the following sequence of surjective morphisms 
for each $e\ge 1$ with $e_0|e$: 
\begin{align*}
\mathcal I^{(e)} \otimes \mathcal L^n \otimes \mathcal A
\twoheadleftarrow &
\left( {F_Y^e}_* \left( f_*\mathcal O_X((p^e-1)N+M) 
\otimes \mathcal A^{p^e} \right) \right) \otimes \mathcal L^n
\\ \cong & 
\left( {F_Y^e}_* \left( f_*\mathcal O_X((p^e-1)N+M) \otimes \mathcal A^l
\otimes \mathcal A^{p^e-l} \right) \right) \otimes \mathcal L^n
\\ \twoheadleftarrow & 
\left( {F_Y^e}_* \left(\bigoplus \mathcal O_Y \right)
\otimes \mathcal A^{p^e-l} \right) \otimes \mathcal L^n
\\ \cong & 
\bigoplus\left({F_Y^e}_*\mathcal A^{p^e-l} \right) \otimes \mathcal L^n
\end{align*}
When $e\gg0$, we see that the last sheaf is 
$0$-regular with respect to $\mathcal L$ 
by an argument similar to the proof of Theorem~\ref{thm:main-p1},
so it is globally generated, and hence so is 
$
\mathcal I^{(e)} \otimes \mathcal L^n \otimes \mathcal A.
$
\end{proof}
Finally, we prove Theorem~\ref{thm:main-0}. 
\begin{proof}[Proof of Theorem~\ref{thm:main-0}]
We first prove the statement on $\mathcal J_{\mathrm{NLC}}(X,\Delta)$. 
One can easily check that we may assume that 
the base field is an algebraically closed field. 
Let $\pi:Z\to X$ be a resolution of $(X,\Delta)$ 
with $K_Z+\Delta_Z=\pi^*(K_X+\Delta)$ 
such that $\mathrm{Supp}(\Delta_Z)$ is simple normal crossing. 
Put 
$
\Delta':=\{\Delta_Z\} +\Delta_Z^{=1}.
$ 
Then each coefficient in $\Delta'$ is at most one 
and $\mathrm{Supp}(\Delta')$ is simple normal crossing. 
Set $M':=\pi^*M - \lfloor \Delta_Z \rfloor +\Delta_Z^{=1}$. 
Then 
\begin{align*}
M' -(K_Z+\Delta') 
= & \pi^*M - \lfloor \Delta_Z \rfloor +\Delta_Z^{=1} 
-K_Z -\{\Delta_Z\} -\Delta_Z^{=1}
\\ = & \pi^*M -K_Z -\Delta_Z
=\pi^*(M-(K_X+\Delta)), 
\end{align*}
so $M'-(K_Z+\Delta')$ is nef and $g$-semi-ample, where $g:=f\circ\pi:Z\to Y$.
We also have 
$
\pi_*\mathcal O_Z(M') 
\cong \mathcal J_{\mathrm{NLC}}(X,\Delta)(M)
$
by the projection formula. 
Let $L$ (resp. $A$) be a Cartier divisor on $Y$ such that 
$\mathcal O_Y(L)\cong \mathcal L$ (resp. $\mathcal O_Y(A)\cong \mathcal A$). 
Put $L^{(i)}:=(n-i)L +\frac{1}{2}A$ for each $0<i\le n$. 
Then each $L^{(i)}$ is ample, and  
$M'-(K_Z+\Delta') +g^*L^{(i)}$ is semi-ample. 
Indeed, since $M'-(K_Z+\Delta')$ is nef and $g$-semi-ample, 
there are projective morphisms $\sigma:X\to W$ and $\tau:W\to Y$ 
with $\tau\circ \sigma=g$ such that 
$M'-(K_Z+\Delta') \sim_{\mathbb Q}\sigma^*N$ 
for a nef and $\tau$-ample $\mathbb Q$-Cartier divisor $N$, 
and then $N+\tau^*L^{(i)}$ is ample, 
so $M'-(K_Z+\Delta')+g^*L^{(i)}$ is semi-ample. 
Therefore, we can find an effective $\mathbb Q$-divisor 
$F^{(i)}\sim_{\mathbb Q} M'-(K_Z+\Delta')+g^*L^{(i)}$ such that 
the support of $\Delta^{(i)}:=\Delta'+F^{(i)}$ is simple normal crossing and 
that each coefficient in $\Delta^{(i)}$ is at most one. 
Then 
\begin{align*}
& M' +g^*((n-i)L+A)-\left(K_Z+\Delta^{(i)}\right)
\\ = & M' +g^*((n-i)L+A)-(K_Z+\Delta') -F^{(i)}
\\ \sim_{\mathbb Q} &
M' +g^*((n-i)L+A)-(K_Z+\Delta') -M'+K_Z+\Delta' -g^*L^{(i)}
\\ = & 
g^*\left((n-i)L+A -(n-i)L -\frac{1}{2}A \right)
=g^*\left( \frac{1}{2}A \right),
\end{align*}
so we can apply \cite[Theorem~3.2]{Amb03} or \cite[Theorem~6.3]{Fuj11}, 
from which we obtain that 
\begin{align*}
0 = & H^i\big(Y, g_*\mathcal O_Z\big(M'-g^*((n-i)L+A) \big)\big)
\\ \cong & H^i\big(Y, 
f_*\big(\mathcal J_{\mathrm{NLC}}(X,\Delta)\big(M-f^*((n-i)L+A) \big) \big) \big)
\\ \cong & H^i\left(Y, f_*(\mathcal J_{\mathrm{NLC}}(X,\Delta)(M))
\otimes \mathcal L^{n-i} \otimes \mathcal A \right) 
\end{align*}
for each $0<i\le n$. 
This implies that 
$
f_*(\mathcal J_{\mathrm{NLC}}(X,\Delta)(M))
\otimes\mathcal L^n \otimes \mathcal A
$ 
is $0$-regular with respect to $\mathcal L$ 
in the sense of Castelnuovo--Mumford, 
and hence it is globally generated (\cite[Theorem~1.8.5]{Laz04I}). 
 
The statement on $\mathcal J(X,\Delta)$ can be proved by an argument similar to the above, by putting $\Delta':=\{\Delta_Z\}$ and $M':=\pi^*M-\lfloor\Delta_Z\rfloor$. 
\end{proof}
\section{Question} \label{section:question}
In this section, we consider the following question:
\begin{ques} \label{ques:Fujita}
Let the base field be an algebraically closed field of characteristic zero. 
Let $X$ be a normal projective variety and 
let $\Delta$ be an effective $\mathbb Q$-Weil divisor on $X$ such that 
$(X,\Delta)$ is log canonical. 
Let $M$ be a Cartier divisor on $X$. 
Let $f:X\to Y$ be a surjective morphism to 
a smooth projective variety $Y$ of dimension $n$. 
If $M-(K_X+\Delta)$ is nef and $f$-semi-ample, then is 
$$
f_*\mathcal O_X(M) \otimes \mathcal L^{n+1}
$$
generated by its global sections for an ample line bundle $\mathcal L$ on $Y$?
\end{ques}
This question is a generalization of Fujita's freeness conjecture. 

In positive characteristic, 
Question~\ref{ques:Fujita} has been already answered negatively, 
even if we employ $S^0f_*(\sigma(X,\Delta) \otimes \mathcal O_X(M))$
instead of $f_*\mathcal O_X(M)$. 
Indeed, Gu--Zhang--Zhang~\cite{GZZ22} constructed a smooth projective surface $S$
on which there is an ample Cartier divisor $A$ such that 
$$
f_*\mathcal O_S(K_S) \otimes A^3
\cong S^0f_*(\sigma(S,0)\otimes \mathcal O_S(K_S)) \otimes A^3
\cong \omega_S \otimes A^3
$$
is not globally generated, where $f=\mathrm{id}:S\to S$. 

We answer affirmatively Question~\ref{ques:Fujita} 
when $Y$ is a smooth projective curve. 
\begin{thm} \label{thm:curve}
Let the base field be an algebraically closed field of characteristic $p\ge 0$. 
Let $X$ be a normal projective variety and 
let $\Delta$ be a $\mathbb Q$-Weil divisor on $X$ 
such that $K_X+\Delta$ is $\mathbb Q$-Cartier. 
Let $f:X\to Y$ be a morphism to a smooth projective curve $Y$. 
Let $M$ be a Cartier divisor on $X$ such that 
$M-(K_X+\Delta)$ is nef and $f$-semi-ample, 
and let $\mathcal L$ be an ample line bundle on $Y$. 
\begin{enumerate}[$(1)$]
\item Suppose that $p=0$. Then 
$$
f_*(\mathcal J_{\mathrm{NLC}}(X,\Delta)(M)) \otimes \mathcal L^l
$$
is generated by its global sections for $l\ge 2$.  
\item Suppose that $p>0$ and 
$i(K_X+\Delta)$ is Cartier for an integer $i>0$ not divisible by $p$. 
Then 
$$
\left( S^0f_*(\sigma(X,\Delta) \otimes \mathcal O_X(M)) \right) 
\otimes \mathcal L^l
$$
is generated by its global sections for $l\ge 2$.  
\end{enumerate}
\end{thm}
\begin{proof}
First, we prove (1). 
Take a closed point $y\in Y$. 
It is enough to show that 
$$
H^1(Y, f_*(\mathcal J_{\mathrm{NLC}}(X,\Delta)(M)) \otimes \mathcal L^l (-y)) =0.
$$
This follows from the proof of Theorem~\ref{thm:main-0}, 
since $\mathcal L^l(-y)$ is ample. 

Next, we show (2). By the same argument as the above, 
it is enough to show that 
$$
H^1\left( Y, \left(S^0f_*(\sigma(X,\Delta)\otimes \mathcal O_X(M))\right) 
\otimes \mathcal A \right)
=0
$$
for an ample line bundle $A$ on $Y$. 
By Theorem~\ref{thm:main-p1} and the proof of Theorem~\ref{thm:main-p2}, 
we have the surjective morphism 
$$
\bigoplus {F^e_Y}_* \mathcal A^{p^e-l} 
\twoheadrightarrow 
\left(S^0f_*(\sigma(X,\Delta) \otimes \mathcal O_X(M))\right) \otimes \mathcal A
$$
for an $l\ge 1$ and each $e$ large and divisible enough. 
Hence, it suffices to prove that 
$$
H^1\left(Y, {F_Y^e}_* \mathcal A^{p^e-l} \right) =0, 
$$
but this follows from an argument similar to 
that of the proof of Theorem~\ref{thm:main-p1}. 
\end{proof}
\bibliographystyle{abbrv}
\bibliography{ref.bib}

\begin{thebibliography}{10}

\bibitem{Amb03}
F.~Ambro.
\newblock Quasi-log varieties.
\newblock {\em Proc. Steklov Inst. Math.}, 240:214--233, 2003.

\bibitem{Eji19p}
S.~Ejiri.
\newblock Positivity of anti-canonical divisors and {$F$}-purity of fibers.
\newblock {\em Algebra Number Theory}, 13(9):2057--2080, 2019.

\bibitem{Fuj10}
O.~Fujino.
\newblock Theory of non-lc ideal sheaves: basic properties.
\newblock {\em Kyoto J. Math.}, 50(2):225--245, 2010.

\bibitem{Fuj11}
O.~Fujino.
\newblock Fundamental theorems for the log minimal model program.
\newblock {\em Publ. Res. Inst. Math. Sci.}, 47(3):727--789, 2011.

\bibitem{Fuj19}
O.~Fujino.
\newblock On mixed-$\omega$-sheaves.
\newblock {\em arXiv preprint arXiv:1908.00171}, 2019.

\bibitem{FM21i}
O.~Fujino and S.-i. Matsumura.
\newblock Injectivity theorem for pseudo-effective line bundles and its
  applications.
\newblock {\em Trans. Amer. Math. Soc. Ser. B}, 8(27):849--884, 2021.

\bibitem{GZZ22}
Y.~Gu, L.~Zhang, and Y.~Zhang.
\newblock Counterexamples to {F}ujita's conjecture on surfaces in positive
  characteristic.
\newblock {\em Adv. Math.}, 400:108271, 2022.

\bibitem{HX15}
C.~D. Hacon and C.~Xu.
\newblock On the three dimensional minimal model program in positive
  characteristic.
\newblock {\em J. Amer. Math. Soc.}, 28:711--744, 2015.

\bibitem{Laz04I}
R.~K. Lazarsfeld.
\newblock {\em Positivity in Algebraic Geometry {I}}, volume~49 of {\em
  Ergebnisse der {M}athematik und ihrer {G}renzgebiete. 3. {F}olge}.
\newblock Springer-Verlag Berlin Heidelberg, 2004.

\bibitem{MB81}
L.~Moret-Bailly.
\newblock Familles de courbes et de vari\'et\'es ab\'eliennes sur {$\mathbb
  P^1$}.
\newblock In {\em Ast\'erisque}, volume~86, pages 125--140. Soci\'e t\'e
  Math\'e matique de France, 1981.

\bibitem{Pat14}
Z.~Patakfalvi.
\newblock Semi-positivity in positive characteristics.
\newblock {\em Ann. Sci. \'Ecole Norm. Sup}, 47(5):991--1025, 2014.

\bibitem{Sch14}
K.~Schwede.
\newblock A canonical linear system associated to adjoint divisors in
  characteristic {$p>0$}.
\newblock {\em J. Reine Angew. Math.}, 696:69--87, 2014.

\bibitem{SZ20}
J.~Shentu and Y.~Zhang.
\newblock On the simultaneous generation of jets of the adjoint bundles.
\newblock {\em J. Algebra}, 555:52--68, 2020.

\bibitem{Zha19s}
L.~Zhang.
\newblock Subadditivity of {K}odaira dimensions for fibrations of three-folds
  in positive characteristics.
\newblock {\em Adv. Math.}, 354:106741, 2019.

\end{thebibliography}
\end{document}